\documentclass[a4paper,12pt]{amsart} 
\title{Distributions of rational points on Kummer Varieties}

\usepackage{amsmath, amssymb, amsfonts, url,mathrsfs, tikz, wrapfig, newclude, amsthm}
\usepackage{hyperref}
\usepackage[all]{xy}

\DeclareMathOperator{\Pic}{Pic}

\DeclareMathOperator{\Gal}{Gal}

\def\Q{\mathbb{Q}}
\def\Z{\mathbb{Z}}

\newcommand{\Span}[1]{\left<#1\right>}

\newcommand{\algcl}{^{\textrm{alg}}}

\usepackage[OT2,OT1]{fontenc}
\DeclareSymbolFont{cyrletters}{OT2}{wncyr}{m}{n}
\DeclareMathSymbol{\Sha}{\mathalpha}{cyrletters}{"58}

\newcommand{\bbF}{{\mathbb F}}

\newcommand{\bbP}{{\mathbb P}}
\newcommand{\bbQ}{{\mathbb Q}}

\newcommand{\bbZ}{{\mathbb Z}}

\newcommand{\cC}{{\mathcal C}}

\newcommand{\ra}{\rightarrow}

\newtheorem{definition}{Definition}[section]
\newtheorem{proposition}[definition]{Proposition}
\newtheorem{lemma}[definition]{Lemma}
\newtheorem{theorem}[definition]{Theorem}
\newtheorem{claim}[definition]{Claim}

\newtheorem{conjecture}[definition]{Conjecture}
\newtheorem{condition}[definition]{Condition}

\def\defeq{\stackrel{\text{\tiny def}}{=}}

\author{David Holmes and Ren\'e Pannekoek}
\date{\today}

\address{Mathematisch Instituut\\ 
Universiteit Leiden\\
Postbus 9512\\
2300 RA Leiden\\
Netherlands}

\email{d.s.t.holmes@umail.leidenuniv.nl, rene@math.leidenuniv.nl}


\newcounter{nootje}
\newcommand{\beq}{\begin{equation}}
\newcommand{\eeq}{\end{equation}}
\newcommand{\beqs}{\begin{equation*}}
\newcommand{\eeqs}{\end{equation*}}
\renewcommand{\k}{k}
\renewcommand{\l}{l}
\newcommand{\hh}{\operatorname{h}}
\newcommand{\rk}{\operatorname{rk}}

\begin{document}
\maketitle
\begin{abstract}
We prove several results on the number of rational points on open subsets of Kummer varieties of arbitrary dimension. Some of our results are unconditional, and others depend on the Parity Conjecture (a corollary of the Conjecture of Birch and Swinnerton-Dyer). As examples, we show that (conditional on the Parity Conjecture) all odd-dimension Kummer varieties over $\bbQ$ which are quotients of absolutely simple abelian varieties have dense rational points, and we construct an infinite family of K3 surfaces with dense rational points. 
\end{abstract}
\tableofcontents

\section{Introduction}

In this paper, we prove several lower bounds for the number of rational points on Kummer varieties over number fields. All of our bounds are valid on any open subset of the Kummer, and so we are not just detecting points on some lower-dimensional subvariety. Most of our bounds are polynomial in nature, though some are exponential (see Section \ref{sec:notation} for the statements). 

The case of dimension 1 is of course uninteresting. It is also exceptional, in so much as it is the only dimension in which Kummer varieties have negative Kodaira dimension (the dimension being zero in higher dimensions). Relatively few results are known on the distribution of rational points in dimensions 3 and above, outside the case of abelian varieties and rational varieties; we make a small contribution to this. 

In dimension 2, Kummer varieties are examples of K3 surfaces, whose rational points are widely studied. In general, Kummer surfaces are not elliptic fibrations (over the ground field), but we do not know whether they may have infinite automorphism groups. Thus it is not clear whether they are covered by the potential-density results  of Bogomolov and Tschinkel (\cite{bogomolov1999density}). 

The authors wish to thank Jan Steffen M\"uller for suggesting this problem and for many helpful discussions along the way. They also wish to thank Bas Edixhoven, Hendrik Lenstra, Frans Oort and Marco Streng for their help with understanding endomorphism rings of abelian surfaces, and Ronald van Luijk for his help with the automorphism groups. 

\section{Notation and statements of main results}\label{sec:notation}

Throughout, $A$ will denote a principally polarised abelian variety with polarising line bundle $\vartheta$ over a number field $k$, and $K = A/\{\pm 1\}$ its (singular) Kummer variety. We also assume that $A$ is simple over every quadratic extension of $k$. Write $\pi:A \ra K$ for the quotient map (of degree 2). 
The polarisation on $A$ induces a canonical height $h$ with respect to $2\vartheta$ (the normalisation will not matter substantially) in the sense of N\'eron-Tate - it is a positive definite quadratic form (in particular, we are working with logarithmic heights). We assume the height to be normalised so that it is invariant under taking algebraic field extensions. Now given a point $p \in K(k\algcl)$, there exists a point $q \in A(k\algcl)$ such that $\pi(q) = p$, and we set $h(p) = h(q)$; note that this is independent of the choice of $q$. 

Given an algebraic variety $X/k$ with a non-degenerate height function $h$, and $l/k$ a finite extension, we set
\beqs
n_{X,l, h}(B) = \#\{ p \in X(l) : h(p) \le B \}. 
\eeqs
If $X$ is a subset of $A$ or of $K$, then we take the height on $X$ to be that induced from the height on $A$ or respectively $K$ as defined above, and we omit it from the notation. Similarly if $l=k$ then we omit it from the notation.

Our main results are as follows:

\begin{theorem}\label{thm1} Given an integer $n>0$, there exists a number field $f/k$ such that on any dense open subset $U \subset K$, there exists a constant $c_U > 0$ with
\beq
n_{U,f}(B) > c_U B^n. 
\eeq
\end{theorem}

\begin{theorem}\label{thm2} Suppose that the Parity Conjecture holds, and suppose that the abelian variety $A$ satisfies one of the conditions given in \ref{conditions_on_A} (for example, $k = \mathbb{Q}$ and $\dim(A)$ is odd). Then for any dense open subset $U$ of $K$, there exists a constant $c_U >0$ such that
\beq
n_U(B) > c_U B^{1/2}. 
\eeq
In particular, rational points are Zariski dense. 
\end{theorem}

\begin{theorem}\label{thm:uncon1}
Let $C$ be a hyperelliptic curve over a number field $k$ with a rational point $p \in C(k)$. Suppose that the Jacobian $J$ of $C$ is simple over every quadratic extension of $k$. Then for any dense open subset $U$ of the Kummer surface $K$ of $J$, there exists a constant $c_U >0$ such that
\beq
n_{U}(B) > c_U B^{{1/2}}. 
\eeq
\end{theorem}

\begin{theorem}\label{thm:uncon2}
Let $C$ be a hyperelliptic curve over a number field $k$, such that $C$ has a $k$-rational Weierstrass point and the Jacobian $J$ of $C$ is simple. Then for any open subset $U$ of the Kummer $K$ of $J$, there exists a constant $c_U>0$ and an integer $n_U>0$ such that
\beq
n_U(B) >c_U \exp(B/n_U). 
\eeq
\end{theorem}

\subsection{Structure of the proof}
We will begin by proving some basic results relating the rational points on an open subset of $K$ to the points on $A$ defined over a quadratic extension of the ground field. Theorems \ref{thm1}, \ref{thm:uncon1} and \ref{thm:uncon2} will follow easily from this. Next, we will discuss the Parity Conjecture, and show that the conditions in \ref{conditions_on_A} are sufficient to force the sign of the functional equation of $A$ to `jump up' in some quadratic extensions, which will imply Theorem \ref{thm2}. 
We then collect some easy consequences of these results, including the construction of an infinite family of K3 surfaces with (conditionally) dense rational points. Finally, we compute the Picard ranks of some Kummer surfaces, and comment on how our results relate to the moduli of K3 surfaces. 

\section{The basic results}

The following lemma is the basic tool for constructing rational points on $K$:
\begin{lemma}\label{lem:q}
Let $\l/\k$ be a quadratic extension, and let $p \in A(\l) $ be such that the image of $p$ in $A(l)/A(k)$ has infinite order. Write $\sigma$ for a generator of the Galois group of $l$ over $k$. Write $q \defeq p - \sigma(p) \in A(\l)$. Then

1) $q$ has infinite order in $A(l)/A(k)$ and maps to a $\k$-rational point on the Kummer $K$;

2) $\hh(q) \le 4\hh(p)$. 
\end{lemma}
\begin{proof}
Clearly $\sigma(q) = -q$ so $q$ maps to a rational point on $K$. 
To prove Assertion (1), it remains to check that $q$ has infinite order in $A(l)/A(\k)$. Well, otherwise there exists an $m>0$ such that $mq = \sigma(mq) = -mq$, so $2mq = 0$, and hence $\sigma(2mp)  = 2mp$, contradicting our assumption that $p$ has infinite order in $A(l)/A(\k)$. 

Assertion (2) follows from the parallelogram law. 
\end{proof}

We recall
\begin{lemma}\label{lem:M-L}
Let $A$ be an abelian variety, and $l/k$ a finite extension such that the base-change $A_{l}$ is simple. Let $Z$ be a proper Zariski-closed subset of $A$. Then $Z(l)$ is finite. 
\end{lemma}
\begin{proof}
This follows immediately from the Mordell-Lang Conjecture, a theorem due to Faltings \cite{FaltingsMordellLang}.  
\end{proof}

Our basic result is then
\begin{proposition}\label{prop:basic}
Let $A$ and $K$ be as above, and $U \subset K$ dense open. Suppose that $l/k$ is a quadratic extension, and set $r_{l} = \rk(A(l))$, $r_{k}  = \rk(A(k))$, $r = \max(r_{k}, r_{l}-r_{k})$. Then there exists a constant $c>0$ such that 
\beq
n_U(B) > c B^{r/2}. 
\eeq
\end{proposition}
\begin{proof}
It is standard that $n_A(B) \gg B^{r_{k}/2}$. By Lemma  \ref{lem:M-L} we know that only finitely many points of $A(k)$ fail to map to points of $U$, and since the canonical map $\pi: A \rightarrow K$ has degree 2, we are done. 

It remains to show that $n_U(B) \gg B^{(r_{l}-r_{k})/2}$. Writing $\sigma$ for the generator of $\Gal(l/k)$ we set 
\beq
S = \{p - \sigma(p) : p \in A(l) \}, 
\eeq
giving an exact sequence
\beq
0 \rightarrow A(k) \rightarrow A(l) \rightarrow S \rightarrow 0, 
\eeq
so $\rk(S) = r_{l}-r_{k}$. Thus 
\beq
\#S(B) \gg B^{(r_{l}-r_{k})/2}, 
\eeq
and by the same arguments as before, we see that only finitely many points of $S$ fail to map to points of $U$, and $\pi$ has finite degree, so we are done. 
\end{proof}


\begin{proof}[Proof of Theorem \ref{thm1}]
Result (1) now follows easily from the observation that for any $r>0$, we can find a finite extension $f/k$ such that $A(f)$ has rank at least $2(r+\rk(A(k))$; indeed, this result needs far less than we have shown above. 
\end{proof}

\begin{proof}[Proof of Theorem \ref{thm:uncon1}]
Write $\pi_1:C \rightarrow \bbP^1$ for the projection map. Any $k$-rational point $x$ on $\bbP^{1}$ will yield a point $q$ on C defined over a quadratic extension, such that $\pi_{1}(q) = x$. The divisor class $[p-q]$ then gives a point on $J$ defined over a quadratic extension. 
However, there are infinitely many such points $x$ and hence infinitely many such points $[p-q]$, and so by Manin-Mumford, they cannot all be torsion, and so we have (infinitely many) points of infinite order defined over quadratic extensions of $k$. The result then follows from Proposition \ref{prop:basic}. 
\end{proof}

Note that this shows that in every dimension, there are infinitely many Kummer varieties where we can show unconditionally that the rational points are Zariski dense. Note also that any hyperelliptic curve aquires a rational point over an at-worst-quadratic extension (which is easy to compute). 

\begin{proof}[Proof of Theorem \ref{thm:uncon2}]
Retaining the notation from the previous proof, and write $p_0$ for a $k$-rational Weierstrass point of $C$ (assumed to exist). We use this point to define an Abel-Jacobi map
\beq
\begin{split}
\alpha: C & \rightarrow J\\
p & \mapsto [p - p_0]. \\
\end{split}
\eeq
Since $p_0$ is a Weierstrass point, the hyperelliptic involution commutes with the involution in the group law on $J$, and so the image $\pi(\alpha(C)) \subset K$ is a copy of $\bbP^1$ with a rational point (for example, the image of $p_0$). We denote this rational curve by $Y$. This already shows that we have exponential growth of rational points on $K$, but now we will make use of the simplicity of $J$ to show that we in fact have exponential growth on every dense open subset. 

Given $n \in \bbZ_{>0}$, we write $[n]:J \rightarrow J$ for the `multiplication by $n$' map. This induces an endomorphism of $K$, which we shall by abuse of notation also denote $[n]$. Consider the set of curves $\{[n]Y :n \in \bbZ_{>0}\}$. As these all have geometric genus zero and posess a rational point, it is enough to show that they cannot all be contained in a proper closed subscheme of $K$. 

Suppose for the purposes of contradiction that this fails; all the $[n]Y$ are contained in some proper closed subscheme $V \subset K$. Write $W \defeq \pi^{-1}V$, a proper closed subscheme of $J$ containing the identity. Arguing as in the previous proof, choose a point $p \in \alpha(C)(\overline{k})$ of infinite order defined over the algebraic closure of $k$. Then the subgroup $<p>$ of $J(\overline{k})$ generated by $p$ is contained in $W$ by construction. By Mordell-Lang, we know that the Zariski and configuration closures of $<p>$ in $J$ coincide. The Zariski closure of $<p>$ is clearly contained in $W$, but the configuration closure must be the whole of $J$ since it is not a finite set of points. Thus $W = J$. 

We now see that every dense open subset of $K$ contains a dense open subscheme of a genus zero curve with a rational point, from which our bounds follow. 
\end{proof}

\section{Parity and ranks}
The Conjecture of Birch and Swinnerton-Dyer predicts that the order of vanishing at $s=1$ of the L-series associated to an abelian variety is equal to its rank. Reducing this modulo 2, we should find that the parity of the rank is equal to the sign of the function equation of the L-series. This statement has the advantage that, although the analytic continuation of the L-series to $s=1$ is only conjectural, the sign of the functional equation (if the continuation exists) is given by the root number, whose existence is not conjectural. Thus it is possible to state the Parity Conjecture, whose statement does not depend on the existence of analytic continuations of L-functions:
\begin{conjecture}\label{conj:parity}
Let $X$ be an abelian variety over a number field, with rank $r$ and global root number $w(X)$ (as defined for example in \cite{dokchitser2009regulator}). Then
\beqs
w(X) \equiv r \mod 2. 
\eeqs
\end{conjecture}
We will apply this to the base change of $A$ to quadratic extensions of $k$. We will control the root number by local considerations, and use this to force its parity to change as we go from $k$ to a quadratic extension. This will force the rank to increase, which was what we wanted. We write
\beq
w(X) = \prod_{\nu} w(X)_{\nu}
\eeq
where the product is over a proper set of absolute values for $k$, and $w(X)_{\nu}$ denotes the local root number at $\nu$, as defined in \cite{dokchitser2009regulator}.

\begin{lemma}
Let $\l/k$ be a finite extension. Fix a place $\nu$ of $k$. 

1) If $\nu$ is non-Archimedean, suppose the N\'eron model of $A$ has split-semistable reduction over $\nu$.  Write $t$ for the rank of the (split) toric part of the connected component of the N\'eron model of $X/\mathcal{O}_{k}$. Then for any place $\nu'$ of $\l$ dividing $\nu$, the local root number is given by
\begin{equation}
w(A/\l)_{\nu'} = (-1)^{t}. 
\end{equation}

2) If $\nu$ is Archimedean, then 
\begin{equation}
w(A/\l)_{\nu'} = (-1)^{\dim(A)}. 
\end{equation}
\end{lemma}
\begin{proof}
1) Applying the formula in \cite[Proposition 3.23]{dokchitser2009regulator} to the trivial representation yields
\begin{equation}
w(A/\l)_{\nu'} = (-1)^{\Span{1, X(\mathcal{T}^{*})}}, 
\end{equation}
where $X(\mathcal{T}^{*})$ is the character group of the dual of the (split) toric part of the connected component of the identity of the N\'eron model of $A/\l$. This is a direct sum of $t$ copies of the trivial representation (since the rank is invariant under finite extensions, and the torus is split), and hence the inner product evaluates to $t$ as desired. 

2) This follows immediately from \cite[Lemma 3.1.1, p31]{sabitova2005root}. 
\end{proof}

In order to apply this result to force the global root number to change sign, we need to impose certain conditions on our abelian variety. 
\begin{condition}\label{conditions_on_A}
The abelian variety $A/k$ is required to satisfy at least one of the following conditions:

1) $A$ has at least one non-Archimedean place with split-semistable reduction and where the rank of the toric part of the N\'eron model is odd;

2) $\dim(A)$ is odd, and $k$ has at least one real place. 
\end{condition}

\begin{proposition}\label{prop:splitting}
Let $S_{1}$, $S_{2}$ be two finite disjoint sets of places of $k$ such that $S_{1}$ does not contain any complex places. Then there exists an infinite set  $T(S_1, S_2)$ of quadratic extensions $\l = k(\sqrt{d_l})$ such that

1) $\forall \nu \in S_{1}$ and $\l \in T(S_1, S_2)$, $\nu$ splits in $\l$;

2)  $\forall \nu \in S_{2}$ and $\l \in T(S_1, S_2)$, $\nu$ does not split in $\l$; 


\end{proposition}
\begin{proof}
Whether a prime $\nu$ splits in $\l$ depends on whether $d_\l$ is a non-zero square in the residue field of $\nu$ (for non-Archimedean $\nu$), or whether $\nu(d_\l)$ is positive (for real Archimedean $\nu$).
\end{proof}

\begin{proposition}
Assume the Parity Conjecture holds for the base-change of $A$ to any quadratic extension of $k$, and that $A$ satisfies one of the conditions in \ref{conditions_on_A}. 
Then there exists a quadratic extension $l/k$ satisfying $\rk(A(k)) < \rk (A(l))$. 
\end{proposition}
\begin{proof}
By assumption, $A$ satisfies one of the conditions in \ref{conditions_on_A}. As such, there exists a place $\nu_0$ of $k$ such that  for all $l/k$, and for all places $\omega $ of $l$ dividing $\nu_0$, we have that $w(A/l)_{\omega}$ is odd. We also have a finite set $S_2$ of places of $k$ such that for all places $\nu$ of $k$ outside $S_2 \cup \{\nu_0\}$, and for all finite $l/k$ and places $\omega$ of $l$ dividing $\nu$, we have that $w(A/l)_{\omega}$ is even. 

Applying Proposition \ref{prop:splitting}, we obtain an infinite set $\tilde{T}$ of quadratic extensions such that in each extension, $\nu_0$ splits and no place in $S_2$ splits. As such, for each $l \in \tilde{T}$, we find that the \emph{global} root number of $A$ over $l$ is different from the global root number of $A$ over $k$. Under the parity conjecture, this forces the rank of $A$ over each such $l$ to be different from the rank over $k$; the rank cannot decrease, so it must increase, as desired. Picking any $l$ in $\tilde{T}$, we are done. 
\end{proof}


From this and Proposition \ref{prop:basic}, Theorem \ref{thm2} follows immediately. 

\section{Examples}
In this section, we note some simple consequences of the above results, and construct some explicit examples. 

\subsection{Consequences of unbounded ranks}
It is an open question whether the ranks of quadratic twists of a fixed abelian variety are bounded. Suppose that the ranks of quadratic twists of $A$ are unbounded. Then on any dense open subset $U \subset K$, and for any $n>0$, there exists a constant $c_{U,n} > 0$ such that 
\beq
n_U(B) > c_{U,n} B^n. 
\eeq

To deduce this, it suffices to recall that quadratic twists $A^l$ of $A$ are constructed by descent from the base-change of $A$ to quadratic extensions $l/k$, and that in this setup we have $\rk(A(k)) + \rk(A^l(k)) = \rk(A(l))$. In particular, the conjecture implies that the rank of $A$ over quadratic extensions $l/k$ is unbounded, and so the result follows immediately from Proposition \ref{prop:basic}. 

\subsection{Odd dimension}
An odd-dimensional Kummer variety that is the quotient of an abelian variety simple over all quadratic extensions clearly satisfies Condition (2), and so (assuming the Parity Conjecture) has rational points growing at least like $c B^{1/2}$ on any dense open subset. Absolutely simple abelian varieties are Zariski dense in the moduli space of polarised abelian varieties, and hence we may deduced from the Parity Conjecture that `most' odd dimensional Kummers have rational points growing at least like $c B^{1/2}$ on any dense open subset. 


\subsection{Dimension 2}
Desingularised Kummer surfaces are examples of K3 surfaces. As we are concerned with rational points on open subsets, we may disregard the desingularisation, and so our results have consequences for the distribution of rational points on these K3 surfaces. 

Since the dimension is even, we must use Condition (2) to construct examples. The reduction type is a local condition, and we will see below that simplicity of $A$ can also be enforced locally, allowing us to easily construct infinite families of K3 surfaces whose rational points grow at least like $cB^{1/2}$, assuming the Parity Conjecture. We give one such construction in the remainder of this note. 

\subsubsection{The construction}
Let $p$ be a prime not dividing 66. Let $C$ be an even-degree hyperelliptic curve of genus 2 over $\bbQ$ given by an equation
$y^{2} = f(x)$, with $f \in \bbZ[x]$ a monic integral polynomial with non-vanishing discriminant $\Delta$. Suppose $f$ is of the form
\beq
x^6 + ax^2 + b
\eeq
where $a \in 1 + p \bbZ_p$, $b \in p + p^2 \bbZ_p$ and $a$, $b \in 1 + 11 \bbZ_{11}$ (the reasons for these conditions will become clear later on).  

Let $A$ denote the Jacobian of $C$, and $K$ the corresponding singular Kummer surface. Following \cite{flynn_and_smart}, $K$ may be embedded in $\bbP^{3}$ with coordinates $k_1$, $k_2$, $k_3$, $k_4$ as a quartic hypersurface given by the equation
\beqs
(k_2^2 - 4k_1k_3)k_4^2 + (-4k_1^3 b +2 k_1^2k_3 a + 2k_3^3)k_4 -4k_1^4 a b - 4 k_1^2k_3^2b +8k_1k_2^2k_3b - 4k_2^4b -4k_2^2k_3^2a. 
\eeqs
We remark that our height $h$ coincides with the naive height on $\bbP^3$ up to $O(1)$, as noted in \cite{flynn_and_smart}. 

\newcommand{\End}{\operatorname{End}}

\begin{claim}\label{claim:simple}
The abelian variety $A$ is simple over every quadratic extension $l/\bbQ$. 
\end{claim}
\begin{proof}
We immitate the proof of Proposition 9 of \cite{FlynnPoononSchaefer}; see also \cite{stoll1996example}. We see that $C$ and hence $A$ has good reduction over $p=11$, and hence the reduction map embedds $\End A_{p}$ into $\End A_{\overline{\bbF}_{p}}$. The characteristic polynomial of the Frobenius endomorphism $\phi$ on $A$ is given by
\beq
P = X^{4} - tX^{3} + sX^{2} - ptX + p^{2} = X^4 + 4X^3 - 132X^2 + 44X + 121, 
\eeq
where
\beq
t = p + 1 - \#C(\bbF_{p}) = -4, 
\eeq
and
\beq
s = \frac{1}{2}((\#C(\bbF_{p}))^{2} - \#C(\bbF_{p^{2}})) +p -(p+1)\#C(\bbF_{p}) = -132. 
\eeq
It is easy to check that this characteristic polynomial is irreducible, and so $A$ is simple and $\bbQ(\phi) \defeq \bbQ[X]/P(X)$ is a field. To check that $A$ remains simple in every quadratic extension of $\bbQ$, we could repeat this calculation replacing $\bbF_{p}$ by its unique degree-2 extension, but it is perhaps easier (and equivalent) to check that $\bbQ(\phi^{2})$ does not lie in a proper subfield of $\bbQ(\phi)$, which can easily be checked using Sage, MAGMA or similar. 
\end{proof}

\begin{claim}
The abelian variety $A$ has split semistable reduction over $q$, with odd-dimensional toric part. 
\end{claim}
\begin{proof}
We consider the projective closure $\cC$ of the curve $C$ over $\bbZ_q$. 
The only non-smooth point is the node at the ideal $(x,y,q)$, and $\cC$ is regular at that point since the dimension of the tangent space to that ideal is two. The node has rational tangent directions $\pm 1$, and hence the reduction is split semistable. To see that the rank of the toric part is odd, we note that the dual graph of the special fibre is a loop, and hence by \cite[Chapter 9, Example 8, Page 246]{bosch1990neron} that the toric part has rank 1. 
\end{proof}

Thus we may conclude that, assuming the Parity Conjecture, all of the K3 surfaces constructed above have rational points growing faster than $cB^{1/2}$ on every dense open subset; in particular, rational points are Zariski dense.

\section{Picard ranks of Kummer surfaces}
When studying the arithmetic of surfaces, the Picard rank and geometric Picard rank are important invariants, appearing for example in the conjecture of Batyrev-Manin-Peyre. In this section, we will discuss the computation of the Picard and geometric Picard ranks for Kummer surfaces, and compute them for generic Kummer surfaces. 

\subsection{Geometric Picard ranks}
We begin by discussing the geometric Picard lattice and its rank. For this, our main reference is the paper \cite{elsenhans2009computation} of Elsenhans and Jahnel. Recall that the geometric Picard rank of a K3 surface ranges between 1 and 20. Write $\tilde{K}$ for the desingularisation of the Kummer variety $K$, obtained by blowing up the 16 nodes which are the images of 2-torsion points in $A$. From \cite[Fact 4.1]{elsenhans2009computation} we have
\beq
\rk \Pic \tilde{K}_{\overline{\bbQ}} = \rk NS(A_{\overline{\bbQ}}) + 16. 
\eeq
Now $NS(A_{\overline{\bbQ}})  \otimes_{\bbZ} \bbQ$ is isomorphic to the vector subspace of Rosati-invariant elements of the endomorphism ring $End^0(A_{\overline{\bbQ}})$ (see \cite[Section 20, page 189]{mumford_ab_vars}). It seems to be widely believed that the `generic' case is that $End^0(A) = \bbQ$, but the authors are not aware of a proof (or even a precise statement) of this in the literature. A precise statement and proof would take us rather far from the main thrust of this paper, and so will not be given here; instead, this will be treated in a forthcoming paper.


\renewcommand{\Pic}{\textnormal{Pic}}
\newcommand{\Kum}{\textnormal{Kum}}
\newcommand{\h}{\textnormal{H}}
\newcommand{\PGL}{\textnormal{PGL}}
\newcommand{\K}{\widetilde{K}}
\newcommand{\W}{\mathcal{W}}
\newcommand{\ol}[1]{\overline{#1}}
\newcommand{\Aut}{\operatorname{Aut}}
\newcommand{\wt}[1]{\widetilde{#1}}
\newcommand{\rank}{\operatorname{rk}}

\subsection{Picard ranks over the ground field}
Write $G_k$ for the absolute Galois group of $k$. Let $C$ be a genus 2 curve with $\W = \{w_1,\ldots,w_6\}$ its set of Weierstrass points, let $J$ be its Jacobian and $\K$ the desingularised Kummer surface of $J$. 
The Picard lattice of $\tilde{K}$ has finite index in the Galois-invariant part of the geometric Picard lattice, and so its rank may be computed given sufficient information on the Galois action. 

 The lattice $E$ generated by the exceptional curves on $\wt{K}$ is a Galois permutation module:
\beq
E = \Z e_1 \oplus \bigoplus_{1 \leq i \leq 5} \Z e_{i6} \oplus \bigoplus_{1 \leq i < j \leq 5} \Z e_{ij} 
\eeq
where the divisor $e_1 \subset \K$ corresponds to the identity of $J$, and the divisors $e_{ij} \subset \K$ each correspond to the 2-torsion point $[w_i - w_j]$ on $J$. We set:
\beq
f_1 = e_1, ~~ f_2 = \sum_{1 \leq i \leq 5} e_{i6},~~ f_3 = \sum_{1 \leq i < j \leq 5} e_{ij}.
\eeq
Generically, the image of $G = \Gal(\ol{k}/k)$ in $\Aut(\W) = S_6$  is as large as possible, i.e. equal to the whole group. If we impose the additional condition that $w_6$ be rational, the image of $G$ in $\Aut(\W)$ is $S_5$ generically. We claim that we have: $\rank E^{S_6} = 2$ and $\rank E^{S_5} = 3$. Indeed we have $\geq$ in both cases, since $f_1,f_2+f_3 \in E^{S_6}$ and $f_1,f_2,f_3 \in E^{S_5}$. It follows from the claim that, if the N\'eron-Severi rank is 1,  one has $\rank \Pic(\wt{K}) = 3$; furthermore, in the generic case such that $w_6$ is defined over $k$, one has $\rank \Pic(\wt{K}) = 4$.

We count the number of copies of the trivial representation $\mathbf{1}$ in $E \otimes \Q$ (considered successively as a $\Q[S_5]$- and a $\Q[S_6]$-module), by applying the formula
\beq
\textnormal{number of copies of $\mathbf{1}_{G}$ in $\Q[G]$-module $V$} = \frac{1}{\# G} \sum_{g \in G} \operatorname{Tr}_V(g).
\eeq
On a $\Q[G]$-module $V$ with a permutation basis it is easy to determine the trace of an element $g \in G$: it is the number of basis vectors fixed by $g$. Moreover, $\operatorname{Tr}_V$ is constant on conjugacy classes $\Gamma$ of $G$. 

We apply the formula to $G=S_5$, $V= E/(f_1) \otimes \Q$:
\begin{center}
\begin{tabular}{l|l|l|l}
$\Gamma$ 	& size of $\Gamma$ & $\operatorname{Tr}_V(g)$  for $g \in \Gamma$ & $\sum_{g \in \Gamma} \operatorname{Tr}_V(g)$ \\
\hline
\hline
(1)		& 1	& 15		& 15 \\	
\hline
(12)		& 10	& 7		& 70 \\
\hline
(123)		& 20	& 3		& 60 \\
\hline
(1234)		& 30	& 1		& 30 \\
\hline
(12345)		& 24	& 0		& 0 \\
\hline
(12)(34)	& 15	& 3		& 45 \\
\hline
(12)(345)	& 20	& 1		& 20 \\
\hline
\hline
Totals:		& 120	&		& 240
\end{tabular}
\end{center}
This proves that $\dim_{\Q} (E \otimes \Q)^{S_5} = 3$, so $\rank E^{S_5} = 3$.

We next apply the formula to $G=S_6$, $V= E/(f_1) \otimes \Q$:
\begin{center}
\begin{tabular}{l|l|l|l}
$\Gamma$ 	& size of $\Gamma$	& $\operatorname{Tr}_V(g)$ for $g \in \Gamma$ & $\sum_{g \in \Gamma} \operatorname{Tr}_V(g)$ \\
\hline
\hline
(1)		& 1	& 15		& 15 \\	
\hline
(12)		& 15	& 7		& 105 \\
\hline
(123)		& 40	& 3		& 120 \\
\hline
(1234)		& 90	& 1		& 90 \\
\hline
(12345)		& 144	& 0		& 0 \\
\hline
(123456)	& 120	& 0		& 0 \\
\hline
(12)(34)	& 45	& 3		& 135 \\
\hline
(12)(345)	& 120	& 1		& 120 \\
\hline
(12)(3456)	& 90	& 1		& 90 \\
\hline
(123)(456)	& 40	& 0		& 20 \\
\hline
(12)(34)(56)	& 15	& 3		& 45 \\
\hline
\hline
Totals:		& 720	&		& 720
\end{tabular}
\end{center}
This proves that $\dim_{\Q} (E \otimes \Q)^{S_6} = 2$, so $\rank E^{S_6} = 2$.

\subsubsection{Elliptic fibrations}
Many of the examples of K3 surfaces known to have many rational points possess elliptic fibrations, and this elliptic structure plays a major role in constructing rational points. As such, it is interesting to ask whether the Kummer surfaces we consider have elliptic fibrations. 

Recall from \cite{pyatetskii1971torelli} that a K3 surface over a field $k$ is an elliptic fibration if and only if its Picard lattice contains a non-zero element of square zero. By Hasse-Minkowski, this always happens if the Picard rank is at least 5, and so in our case we know that geometrically our surfaces always possess elliptic fibrations. However, over the ground field this does not hold, as we shall show.

We keep the assumption of geometric Picard rank 17. We show that, both in the $S_5$ and $S_6$ cases, the vector space $\Pic(\wt{K}) \otimes \Q$ equipped with the intersection form has no non-zero elements with square zero. We immediately reduce to the $S_5$ case.

In that case, a basis for $\Pic(\wt{K}) \otimes \Q$ is given by $f_1,f_2,f_3$ and a suitable hyperplane section $h$ such that the intersection pairing on  $\Pic(\wt{K}) \otimes \Q$ reads
\beq
\begin{array}{l|llll}
 	& h & f_1 & f_2 & f_3 \\
\hline
h 	& 4 & 0   & 0   & 0   \\
f_1 	& 0 & -2  & 0   & 0   \\
f_2     & 0 & 0   & -10 & 0   \\
f_3     & 0 & 0   & 0   & -20 
\end{array}
\eeq
We are asking whether the smooth quadric $2x^2-y^2-5z^2-10w^2=0$ over $\Q$ has a point $(x_0 : y_0 : z_0 : w_0)$, where $x_0,y_0,z_0,w_0$ may be assumed to be integers that are not all divisible by the same prime. Since the only solutions to $2x^2 \equiv y^2 \pmod{5}$ is $x \equiv y \equiv 0 \pmod{5}$, we have $5 \mid x_0, 5 \mid y_0$, and hence $25 \mid 5z_0^2+10w_0^2$, so $5 \mid z_0^2+2w_0^2$ and hence $5 \mid z_0, 5 \mid w_0$, a contradiction.

\section{Moduli of Kummer surfaces}
A Kummer surface is by definition the quotient of an abelian surface by $\{\pm 1\}$, and an abelian surface is (in characteristic zero) the Jacobian of a genus 2 curve, which can always be written as $y^2 = f$ for a homogeneous polynomial $f$ of degree 6. A Kummer surface has exactly 16 singular points (they are nodes), and can be embedded in $\bbP^3$ as a quartic hypersurface. We recall from \cite[Page 352]{mumford_equations_1} that any Kummer hypersurface can be written (after a linear change of coordinates) in the form
\beqs
\begin{split}
F(a,b,c,d,e) = &a(x^4 + y^4 + z^4 + w^4)  +  b(xyzw) + c(x^2y^2 + z^2w^2) \\
& + d(x^2w^2 + y^2z^2) + e(x^2z^2 + y^2w^2) = 0. 
\end{split}
\eeqs
Such equations define a 4-dimensional family of quartics with no basepoint, and so by Bertini's Theorem a generic element is non-singular, and so cannot be a Kummer. However, whenever a member of this family has one node it acquires 16 of them (by symmetry). Such a node exists if and only if the discriminant of $F(a,b,c,d,e)$ vanishes. Any member of this family with 16 nodes and no higher singlarities is a Kummer surface; this is a special case of \cite[4.4]{elsenhans2009computation}, which says that any quartic with 16 nodes is a Kummer. 
Given a quartic $K$ with 16 nodes, it is a Kummer and hence it comes from the Jacobian of some genus 2 curve. This curve can be recovered from $K$ as in \cite[4.4]{elsenhans2009computation}. 

We thus obtain a moduli space of polarised Kummer surfaces as an open subset of the discriminant locus of the polynomial $F(a,b,c,d,e)$. Given a point in that locus, we can reconstruct a hyperelliptic curve from whose Jacobian it arises. If that Jacobian is simple  (which in general it will be, as can be seen by counting dimensions of moduli spaces), then we can prove it is so using the techniques from the proof of Claim \ref{claim:simple}. Assuming the Parity Conjecture, we then know that rational points grow faster than $B^{1/2}$ on any open subset; in particular, they are dense. We can also easily determine a quadratic field over which the curve acquires a rational point, and so by Theorem \ref{thm:uncon1} we know \emph{unconditionally} that over that field, the surface has rational points growing faster than $B^{1/2}$ on any open subset. Similarly, we can determine an extension of degree at most $6$ over which the curve has a rational Weierstrass point, and so by Theorem \ref{thm:uncon2} we have unconditional exponential growth over that field on any open subset. 

Consider the case where the geometric Picard rank is 17 and the Picard rank is 3. By considering the Galois action on the Weierstrass points, we can compute for a given example the rank of the Galois invariant part of the representation `$E$' considered in the previous section, obtaining information on the Picard rank. If we could also compute the Neron-Severi rank of the Jacobian, we would be able to precisely compute the Picard rank of the desingularisation $\tilde{K}$. We note that Charles has given an algorithm to compute the Picard rank whose termination is conditional on the Hodge conjecture, but unconditional when the ground field is $\bbQ$ \cite{CharlesAlgorithm}.

\bibliographystyle{alpha} 
\bibliography{prebib.bib}

\end{document}